\documentclass[11pt]{amsart}    
\usepackage{fullpage}
\usepackage{amsmath, amssymb, amsthm}
\usepackage{geometry}         
\usepackage{tikz}
\usetikzlibrary{matrix}
\usepackage{amsmath}
\usepackage{amssymb}
\usepackage{tikz}
\usetikzlibrary{matrix,arrows,decorations.pathmorphing}
\usetikzlibrary{shapes.geometric}
\usepackage{tikz-cd}
\usepackage{amsthm}
\usepackage{graphicx}
\usepackage{amssymb}
\usepackage{tikz-cd}
\usetikzlibrary{calc,arrows}
\title{Categorical view of the Partite Lemma\\ in structural Ramsey theory
} 
\author{Sebastian Junge}
\date{}
\address{Department of Mathematics, Cornell University, Malott Hall, Ithaca, NY 14853}
\email{sj676@cornell.edu}
\newtheorem*{definition}{Definition}

\newtheorem*{lemma*}{Lemma}
\newtheorem{theorem}{Theorem}
\newtheorem{lemma}[theorem]{Lemma}
\newtheorem{corollary}[theorem]{Corollary}

\begin{document}
	\setlength{\footskip}{24pt}
	
	\begin{abstract} We construct of the main object of the Partite Lemma as the colimit over a certain diagram. This gives a purely category theoretic take on the  Partite Lemma and establishes the canonicity of the object. Additionally, the categorical point of view allows us to unify the direct Partite Lemma in \cite{2}, \cite{3}, and \cite{10} with the dual Paritite Lemma in \cite{6}.
	\end{abstract}\maketitle 
	\section{Introduction}
Category theory has long been used in Ramsey theory. Leeb created a category theoretic framework for proving Ramsey statements in lecture notes that were recorded in \cite{15} in 1973. In \cite{13} Graham, Rothschild, and Leeb proved a Ramsey theorem for finite vector spaces with this framework.  The influential Ne\v{s}et\v{r}il--R\"{o}dl Theorem proved in \cite{2}, \cite{3}, \cite{10} in the 1970's and 1980's was expressed in the language of category theory . After these results there was a shift back to the language of classes, though categorical ideas were implicit in Solecki's papers \cite{12},\cite{6}, and \cite{5}. In 2015 Gromov advocated for a deeper use of category theory in Ramsey theory in \cite{16}. More recently the work of Masulovic has created more general categorical techniques to prove Ramsey statements, such as in \cite{7}. We follow Gromov's lead and reexamine a key result in the Ne\v{s}et\v{r}il--R\"{o}dl Theorem using a categorical approach.

 Structural Ramsey theory was initiated by the work of Abramason--Harrington in \cite{1} and  Ne\v{s}et\v{r}il--R\"{o}dl in \cite{2},\cite{3}, and \cite{10} expanded upon the Abramson and Harrington result. Recently Ne\v{s}et\v{r}il and Hubi\v{c}ka proved a version of the Ne\v{s}et\v{r}il--R\"{o}dl Theorem for classes of structures with certain closure properties in \cite{11}. This group of theorems is fundamental to structural Ramsey theory, which has seen a revival in the recent years \cite{14},\cite{17},\cite{18},\cite{19},\cite{20}. At the core of these theorems is a result known as the Partite Lemma.   In \cite{6} Solecki gave a dual version of the Partite Lemma. The present paper shows that the objects produced in the Partite Lemma and its dual version are actually canonical category theoretic objects called cocones and colimits. Furthermore, we emphasize that this exploration of the Partite Lemma allows for a unification of the original Partite Lemma and the dual version of it. 
	
	 The main theorem of this paper asserts that a certain categories have cocones and colimits over diagrams that are defined using Hales--Jewett lines.  While the definition of the diagrams uses Hales--Jewett lines, our main theorem do not involve any Ramsey theory. After establishing, in our main theorem, the existence of the colimits, we prove that the object needed for the Partite Lemma is our colimt. All the properties in the conclusion of the Partite Lemma follow directly from this object being a cocone. So our main theorem isolates the mathematical properties of the construction in the Partite Lemma through the ideas of cocones and colimits.  Apart from exhibiting the category theoretic nature of this object, our result shows that its canonicity as colimits are canonical. Additionally, our approach yields a unification proofs of the Partite Lemma  \cite{2}, \cite{3},\cite {10} and the dual Partite Lemma \cite{6}.
	
	We now describe this paper's organization. In Section 2 we give a generalization of structures where we add a category $\mathcal{C}$ to the definition of language and structures. This allows us to unify structures as occurring  in \cite{5} and  dual structures found in \cite{6}. We then define blocks, which are a generalization of objects in \cite{5} and paritite systems in \cite{2}, \cite{3}, and \cite{10}.  In Section 3 we introduce a subcategory of blocks and a diagram in the subcategory using Hales--Jewett lines. Then we prove our main theorem which gives the existence of colimits over these line diagrams.  In Section 4 we turn our attention to Ramsey theory.  We explain how cocones can be used to transfer the Ramsey property and as a consequence we prove the Partite Lemma using our main theorem. We discuss how to prove the Partite Construction, which is the other main proposition in the Ne\v{s}et\v{r}il--R\"{o}dl Theorem. We finish by applying the Partite Lemma to prove the results in \cite{6} and \cite{5} in a unified manner. \\
	
	The author would like to thank S\l awomir Solecki for spending ample time helping refine the presentation of this paper.

	\section{Structures and blocks}
	We give a brief overview of the types of structures used in Ramsey theory to motivate our definition of structures.  In \cite{2}, \cite{3}, and \cite{10}   Ne\v{s}et\v{r}il and R\"{o}dl prove a Ramsey result for linearly ordered hypergraphs.  Solecki expands on this result in 
	\cite{5} by showing a Ramsey statement for linearly ordered structures with standard interpretation of relation symbols and dual interpretation of function symbols. Furthermore in \cite{6} he proves a dual Ramsey result for linearly ordered structures with the standard interpretations function symbols and dual interpretation of relation symbols (for more information on the interpretations in direct and dual structures see section 4.5) . Note that \cite{14} gives a new proof of a special case (linearly ordered structures with interpretations of function symbols only) of results from \cite{6}.  We give a common generalization of the results in \cite{6} and \cite{5}, that is we unify the dual and direct structural Ramsey theory.
	
There exists clear analogies between the structures and arguments in \cite{6} and \cite{5}. Our aim is to make these analogies concrete and unify the idea of these two papers. To do so we formulate a new notion of structure which includes the structures found in \cite{6} and \cite{5}. A crucial point in this new notion is an addition of a category $\mathcal{C}$ to the definition of structures.  The  structures given in \cite{5} will arise  when $\mathcal{C}=\textbf{Fin}$ and the structures in \cite{6} will arise when $\mathcal{C}=\textbf{Fin}^{\textnormal{op}}$.  Next we formulate the notion of blocks which are a generalization of objects in \cite{5} which in turn build on the definition of partite-systems in \cite{2}, \cite{3}, and \cite{10}.
	\subsection{Structures}
	We develop the concept of structures with a category $\mathcal{C}$ by following the standard development of structures. We start by adding a category $\mathcal{C}$ to the definition of  of language. Then we define structures for these new types of languages. Finally we define homomorphisms in the natural way. 
	
	\begin{definition}
		For any category $\mathcal{C}$, a $\mathcal{C}$\textbf{-language} $\mathcal{L}$ is a tuple $(\mathcal{C},\mathcal{L}_{F},\mathcal{L}_{R}, ar_{func},ar_{rel})$ where $\mathcal{L}_{F}$ is a set of function symbols, $\mathcal{L}_R$ is a set of relation symbols, $ar_{func}\colon \mathcal{L}_{F}\to \textnormal{Ob}(\mathcal{C})^2$ assigns the arity of function symbols, and $ar_{rel}\colon \mathcal{L}_{F}\to \textnormal{Ob}(\mathcal{C})$ assigns the arity of relation symbols.
	\end{definition}
	The usual definition of language has arities whose ranges are finite sets instead of objects in a category $\mathcal{C}$. Thus the definition of a \textbf{Fin}-language is the standard definition of a language. Now that we have the definition of language we can define structures.
	\begin{definition}
		An $\mathcal{L}$\textbf{-structure} $\mathsf{X}$ is an object $X\in \text{Ob}(\mathcal{C})$ along with interpretations of the symbols in $\mathcal{L}$ that are implemented as follows,\\
		for each relation symbol $R\in \mathcal{L}$ of arity $r\in\text{Ob}(\mathcal{C})$ the interpretation of $R$ is a set\\ $R^{\mathsf{X}}\subseteq \textnormal{Hom}_{\mathcal{C}}(r,X)$ and \\
		for each function symbol $F\in \mathcal{L}$ of arity $(r,s)\in \text{Ob}(\mathcal{C})^2$ the interpretation of $F$ is a function
		$F^{\mathsf{X}}\colon \textnormal{Hom}_{\mathcal{C}}(X,r)\to \textnormal{Hom}_{\mathcal{C}}(X,s)$.
	\end{definition}
	
	The standard definition of structures lets $X$ be a set and interpretations are functions instead of morphisms. Furthermore if $\mathsf{X}$ is a structure and $F$ is a function symbol of arity $(r,s)$ then under the usual definition of structure $F^{\mathsf{X}}\colon X^r\to X^s$ while in our definition when $\mathcal{C}=\textbf{Fin}$, $F^{\mathsf{X}}\colon r^X\to s^X$. So if $\mathcal{L}$ is \textbf{Fin}-language, then $\mathcal{L}$-structures have relations which are defined in the usual manner and dual interpretations of function symbols. Thus $\mathcal{L}$-structures are defined as in \cite{5}. 
	
	If $\mathcal{L}$ is a $\textbf{Fin}^{\text{op}}$ language then for any structure $\mathsf{X}$ and relation symbol $R$ of arity $r$, $R^{\mathsf{X}}\subset r^X$. If $F$ is a function symbol of arity $(r,s)$ then $F^{\mathsf{X}}\colon X^r\to X^s$. Thus function symbols are defined in the standard way but the relation symbols are interpreted in a dual manner. Thus $\mathcal{L}$-structures are the same as dual structures found in \cite{6}.
	   
	 Next we define homomorphisms in the natural way.
	\begin{definition}
		Let $\mathcal{C}$ be a category and $\mathcal{L}$ be a $\mathcal{C}$-language.  If $\mathsf{X},\mathsf{Y}$ are $\mathcal{L}$-structures and $f\in\textnormal{Hom}_{\mathcal{C}}(X,Y)$, then $f$ is an $\mathcal{L}$-\textbf{homomorphism} if:\\
		for all relation symbols $R\in \mathcal{L}$ with arity $r\in\text{Ob}(\mathcal{C})$ and all $\eta\in \textnormal{Hom}_{\mathcal{C}}(r,X)$, $$R^\mathsf{X}(\eta)\Leftrightarrow R^\mathsf{Y}(f\circ \eta)$$
		and for all function symbols $F\in \mathcal{L}$ with arity $(r,s)\in\text{Ob}(\mathcal{C})^2$ and all $\gamma\in \textnormal{Hom}_{\mathcal{C}}(Y,r)$, $$F^\mathsf{X}(\gamma\circ f)=F^\mathsf{Y}(\gamma)\circ f.$$
	\end{definition}
\subsection{Blocks}
The objects that we consider in our main theorem are blocks. Blocks are a categorical version of objects in \cite{5}. We will use the term blocks instead of objects to avoid confusion with categorical notation. Objects are a generalization of partite-systems which are used in the Partite Lemma.
	\begin{definition}
		Fix a category $\mathcal{C}$ and a $\mathcal{C}$-language $\mathcal{L}$.  A \textbf{block} is a pair $\mathcal{X}=(\mathsf{X},\pi)$ where $\mathsf{X}$ is an $\mathcal{L}$-structure and $\pi\in \textnormal{Hom}_{\mathcal{C}}(X,U)$ for some $U\in \textnormal{Ob}(\mathcal{C})$. 	\end{definition}

	If $\mathcal{X}$ is a block where the morphism $\pi$ has a left inverse we call $\mathcal{X}$ a \textbf{monic block}. An example of a monic block is a structure $\mathsf{X}$ which can be viewed as the block $\mathcal{X}=(\mathsf{X},\textnormal{Id}_{X})$.
	
	We now define morphisms between blocks. 
	\begin{definition}
		Fix a category $\mathcal{C}$ and a $\mathcal{C}$-language $\mathcal{L}$. Suppose $\mathcal{X}=(\mathsf{X},\pi)$ and $\mathcal{Y}=(\mathsf{Y},\rho)$ are blocks so that 
		$\pi\in\textnormal{Hom}_{\mathcal{C}}(X,U)$ and $\rho\in\textnormal{Hom}_{\mathcal{C}}(Y,V)$. A \textbf{block-homomorphism} between $\mathcal{X}$ and $\mathcal{Y}$ is a homomorphism $f\in\textnormal{Hom}_{\mathcal{C}}(X,Y)$ for which there is an $i\in \textnormal{Hom}_{\mathcal{C}}(U,V)$ such that $\rho\circ f=i\circ \pi$.
	\end{definition}
	
	\begin{center}
		\begin{tikzcd}
			X\arrow[r,"f"]\arrow[d, "\pi"]& Y\arrow[d,"\rho"]\\
			U\arrow[r,"i"] & V\\
		\end{tikzcd}
	\end{center}
A block-homomorphism is called a \textbf{block-monomorphism} if it has a left inverse in $\mathcal{C}$.

For the remainder of this paper we will adhere to the following convention. We use the letters $U,V,W,X,Y,Z$ to denote objects in the underlying category $\mathcal{C}$, the letters $\mathsf{W},\mathsf{X},\mathsf{Y}, \mathsf{Z}$ to denote structures with underlying objects $W,X,Y,Z$, and the letters $\mathcal{W},\mathcal{X},\mathcal{Y},\mathcal{Z}$ to denote blocks with first coordinate $\mathsf{W},\mathsf{X},\mathsf{Y},\mathsf{Z}$.

	Let $Bl$ be the category where objects are blocks and morphisms are block-homomorphisms. Let $Bl^{m}$ be the subcategory of $Bl$ with the same objects but $\text{Hom}(Bl^{m})$ is the class of block-monomorphisms. 
	\section{The main theorem}
	In this section we will show that a specific subcategory of blocks has colimits over certain diagrams that are defined using Hales--Jewett lines. This result describes the construction of the Partite Lemma in a purely category theoretic manner. We start by defining the category and diagram that we need for our main theorem. We will then state and prove our main result.
	\subsection{The category $Bl_{i_0}$}
	In this section we define a subcategory of $Bl$ which can be viewed as a local version of $Bl$. We start by defining the morphisms for this category. For this section fix a category $\mathcal{C}$ and a $\mathcal{C}$-language $\mathcal{L}$.
	\begin{definition}
		Suppose $\mathcal{X}=(\mathsf{X},\pi)$ and $ \mathcal{Y}=(\mathsf{Y},\rho)$ are blocks so that 
		$\pi\in\textnormal{Hom}_{\mathcal{C}}(X,U)$ and $\rho\in\textnormal{Hom}_{\mathcal{C}}(Y,V)$.  If $i_0\in \textnormal{Hom}_{\mathcal{C}}(U,V)$, then an $i_0$-\textbf{homomorphism} between $\mathcal{X}$ and $\mathcal{Y}$ is a block-homomorphism $f\in\textnormal{Hom}_{Bl}(X,Y)$ such that $\rho\circ f=i_0\circ \pi$.
	\end{definition}
	
	An $i_0$-\textbf{monomorphism} is an $i_0$-homomorphism with a left inverse.
	
	We will now define the subcategory for our main theorem.  Fix  a morphism\\ $i_0\in \textnormal{Hom}_{\mathcal{C}}(U,V)$ for some $U,V\in \text{Ob}(\mathcal{C})$. We divide the objects of $Bl_{i_0}$ into two types of objects, domain objects and codomain objects. \textbf{Domain objects} are blocks $(\mathsf{X},\pi)$ where the target of $\pi$ is $U$ (the domain of $i_0$)
	and \textbf{Domain objects} are blocks $(\mathsf{Y},\rho)$ where the target of $\rho$ is $V$ (the codomain of $i_0$).
	Morphisms in $Bl_{i_0}$ between a domain object and a codomain object are $i_0$-homomorphisms, morphisms between domain objects are $\text{Id}_{U}$-homomorphisms, and morphisms between codomain objects are $\text{Id}_{V}$-homomorphisms. 
	
	Let $Bl_{i_0}^{m}$ be the subcategory of $Bl_{i_0}$ with the same objects as $Bl_{i_0}$ where all morphisms have a left inverse in $\mathcal{C}$.
	
	\subsection{The line diagram}
	To define the diagram that we need for our main theorem we introduce the notion of combinatorial lines. 
	\begin{definition}
		If $P$ is a set and $N>0$, then a \textbf{line} $\ell$ in $P^N$ is a nonempty $d(\ell)\subseteq N$ along with $\ell_k\in P$ for each $k\in N\backslash d(\ell)$.
	\end{definition}

	If $\bar{e}\in P^N$ and $\ell$ is a line in $P^N$ we say that $\bar{e}\in \ell$ if $e_k=\ell_k$ for all $k\notin d(\ell)$ and $\bar{e}$ is constant on $d(\ell)$. If $\bar{e}\in\ell$ we let $\ell(\bar{e})$ be the constant value of $\bar{e}$ on $d(\ell)$. Note that for every  $\bar{e}\in P^N$ there is a line $\ell$ so that $\bar{e}\in \ell$.
	
	Given the above definition of lines we will construct a diagram. 
	\begin{definition}
		Fix $N>0$, a category $\mathcal{C}$, a $\mathcal{C}$-language $\mathcal{L}$, $i_0\in \textnormal{Hom}(\mathcal{C})$, a monic domain object $\mathcal{X}\in\text{Ob}(Bl^{m}_{i_0})$, and a codomain object $\mathcal{Y}\in \text{Ob}(Bl^{m}_{i_0})$ . Let $J$ be the category with
		$$\textnormal{Ob}(J)=\textnormal{Hom}_{Bl^{m}_{i_0}}(\mathcal{X}, \mathcal{Y})^N\cup 
		\{\ell\colon \ell\textnormal{ is a line in }
		\textnormal{Hom}_{Bl^{m}_{i_0}}(\mathcal{X},\mathcal{Y})^N\}$$
		For every pair $(\bar{e},\ell)$ where $\bar{e}\in\textnormal{Hom}_{Bl^{m}_{i_0}}(\mathcal{X},\mathcal{Y})^N$, $\ell$ is a line in $\textnormal{Hom}_{Bl^{m}_{i_0}}(\mathcal{X},\mathcal{Y})^N$,  and $\bar{e}\in \ell$ we let $\textnormal{Hom}_J(\bar{e},\ell)$ be a set with one morphism which we denote $(\bar{e},\ell)$. All other morphisms in $J$ are identities. Then the  \textbf{line diagram} is the functor $G\colon J\to Bl^{m}_{i_0}$, $$G(\bar{e})=
		\mathcal{X} \textnormal{ if }\bar{e}\in  \textnormal{Hom}_{Bl^{m}_{i_0}}(\mathcal{X},\mathcal{Y})^N$$
		$$G(\ell)=\mathcal{Y} \text{ if }\ell \text{ is a line in }\textnormal{Hom}_{Bl^{m}_{i_0}}(\mathcal{X},\mathcal{Y})^N $$
		and on non-identity morphisms  $G(\bar{e},\ell)=\ell(\bar{e})$.
	\end{definition}
	Note that the definition of the index category $J$ only depends on $N$, $i_0$, $\mathcal{X}$, and $\mathcal{Y}$. Also notice that since $Bl^{m}_{i_0}$ is a subcategory of $Bl_{i_0}$, $G$ can also be considered as a functor $G\colon J\to Bl_{i_0}$. The object of this section is to build a cocone over $G$ in $Bl^{m}_{i_0}$ that is also the colimit over $G$ in $Bl_{i_0}$. In applications we  only use the existence of a cocone over $G$ in $Bl^{m}_{i_0}$. Being a colimit in $Bl_{i_0}$ assures canonicity of the construction.
	
	\subsection{Statement and proof of the main theorem}
	
	The following theorem is the main result of this paper. In Ramsey theoretic applications only the existence of the object from the conclusion of Theorem 1 is used.  While the proof of this theorem is purely categorical, on a technical level we build on arguments going back to \cite{2}, \cite{3}, and \cite {10}.  Our proof is most closely related to the arguments found in \cite{6} and \cite{5}.
	
	\begin{theorem}
		Let $\mathcal{C}$ be a category that has colimits, let $\mathcal{L}$ be a $\mathcal{C}$-language, and let \\
		$i_0\in\text{Hom}(\mathcal{C})$ have a left inverse. Then for each line diagram $G$, $Bl_{i_0}$ has a colimit over $G$ that is also a cocone over $G$ in $Bl^{m}_{i_0}$. \end{theorem}
	
	\begin{proof}
		
		Fix $N>0$, $\mathcal{X}=(\mathsf{X},\pi)$ a monic domain object in $Bl^{m}_{i_0}$ and $\mathcal{Y}=(\mathsf{Y},\rho)$ a codomain object in $Bl^{m}_{i_0}$.  For ease of notation let $\textnormal{Hom}_{Bl^{m}_{i_0}}(\mathcal{X},\mathcal{Y})=P$.  
		
		Let $H\colon Bl_{i_0}\to \mathcal{C}$ be the functor defined by $H(\mathcal{W})=W$ and $H(f)=f$. Then there is a colimit $(Z,f_{\ell},f_{\bar{e}})_{\ell, \bar{e}\in\text{Ob}(J)}$ in $\mathcal{C}$ over the diagram $H\circ G$ by assumption.  Since the forgetful functor creates colimits for slice categories (see \cite[p.91-92]{9}), there are interpretations of function symbols $F^\mathsf{Z}$ and  a morphism
		$\sigma\in\text{Hom}_{\mathcal{C}}(Z,V)$ so that if $\mathcal{L}$ has no relation symbols, then $(Z,f_{\ell},f_{\bar{e}})_{\ell, \bar{e}\in\text{Ob}(J)}$ is the colimit in $Bl_{i_0}$.
		
	Thus it remains to show that each $f_{\ell}$ is a split-monomorphism and to define appropriate interpretations  $R^{\mathsf{Z}}$ for all relation symbols in $\mathcal{L}$. To do so we need a cocone $(Y, u_{\ell,i},\bar{e}_i)_{\ell,\bar{e}\in \textnormal{Ob}(J)}$ for each $i< N$. Fix $h\in \text{Hom}_{\mathcal{C}}(V,X)$ so that $h\circ i_0\circ \pi=\rm{Id}_{X}$. Such $h$ exists by assumption. For each $i< N$ let $$ u_{\ell,i}=\left\{\begin{array}{ll}
			\rm{Id}_Y & \text{if } i\in d(\ell)\\
			\ell_i\circ h\circ \rho & \text{otherwise}
		\end{array}\right.$$
		We show that $(Y, u_{\ell,i},\bar{e}_i)_{\ell,\bar{e}\in \textnormal{Ob}(J)}$ is a cocone over $H\circ G$. Fix $\ell$ and $\bar{e}$ so that $\bar{e}\in \ell$. Then we prove $u_{\ell,i}\circ \ell(\bar{e})=\bar{e}_i$ by cases. If $i\in d(\ell)$, then $\textnormal{Id}_Y\circ \ell(\bar{e})=\bar{e}_i$ by definition. If $i\notin d(\ell)$, then $\ell_i=\bar{e}_i$, so since $\ell(\bar{e})$ is an $i_0$-homomorphism, $$\ell_i\circ h\circ \rho\circ \ell(\bar{e})=\ell_i\circ h\circ i_0\circ \pi=\ell_i=\bar{e}_i.$$
		Thus  $(Y, u_{\ell,i},\bar{e}_i)_{\ell,\bar{e}\in \textnormal{Ob}(J)}$ is a cocone over $H\circ G$ for all $i<N$. So for all $i<N$ there is a $v_i$ so that for every line $\ell$ in $P^N$ and every $\bar{e}\in P^N$, $v_i\circ f_\ell=u_{\ell,i}$ and $v_i\circ f_{\bar{e}}=e_{i}$. If $\ell$ is a line in $P^N$ then since $d(\ell)$ is nonempty there is $i\in d(\ell)$ so, $$v_i\circ f_{\ell}=u_{\ell,i}=\textnormal{Id}_{Y}.$$
		Thus we have shown that each $f_{\ell}$ has a left inverse in $\mathcal{C}$.
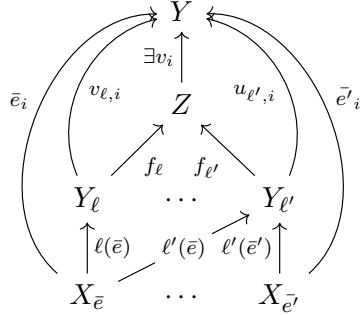
		\begin{equation*}\begin{tikzcd}[column sep=small] &Y\\&Z\arrow[u,"\exists v_i"] \\ Y_{\ell} \arrow[ur, swap, "f_\ell"]\arrow[uur,crossing over, bend left=50, swap, "v_{\ell,i}"] &\cdots& Y_{\ell'} \arrow[ul, "f_{\ell'}"] \arrow[uul,crossing over, bend right=50, "u_{\ell',i}"]\\ X_{\bar{e}} \arrow[u, swap, "\ell(\bar{e})"]\arrow[uuur,bend left=70, "\bar{e}_i"]\arrow[urr, "\ell'(\bar{e})" description]  &\cdots & X_{\bar{e'}} \arrow[u, "\ell'(\bar{e}')"]\arrow[uuul,bend right=70, swap, "\bar{e'}_i"] \end{tikzcd} \label{eq 5}
		\end{equation*}
		
		We define the interpretations of relation symbols on $Z$ as follows. For each relation symbol $R$ in  $\mathcal{L}$ of arity $r$, let $R^{\mathsf{Z}}$ be such that:
		$$R^\mathsf{Z}(\delta)\Leftrightarrow \text{there are } \ell\text{ a line in }P^N \text{ and } \eta\in\textnormal{Hom}_{\mathcal{C}}(Y,r), \text{ so that } \delta=f_{\ell}\circ \eta \text{ and } R^\mathsf{Y}(\eta).$$
		
		Clearly $R^\mathsf{Y}(\eta)$ implies $R^{\mathsf{Z}}(f_\ell \circ \eta)$. So to show that $f_\ell$ is a homomorphism it remains to prove that  $R^\mathsf{Z}(f_\ell \circ \eta)$ implies $R^\mathsf{Y}(\eta)$. So suppose $R^{\mathsf{Z}}(f_\ell\circ \eta)$, then by definition there is a $\eta'\in \text{Hom}_{\mathcal{C}}(Y,r)$ and a line $\ell'$ so that $R^\mathsf{Y}(\eta')$ and $f_{\ell'}\circ\eta'=f_{\ell}\circ \eta$. We will let $f_{\ell}\circ \eta=\delta$. We show that $R^\mathsf{Y}(\eta)$ by two cases. If there is  $i\in d(\ell)\cap d(\ell')$, then $$v_i\circ \delta=u_{\ell,i}\circ \eta=u_{\ell',i}\circ \eta'.$$
		Thus $\eta=\eta'$ by the definition of $u_{\ell,i}$.
		
		In the second case there is no $i\in d(\ell)\cap d(\ell')$ so let $i\in d(\ell)$ and $j\in d(\ell')$. Then by the definition of $v_{i}$,
		$$v_i\circ \delta=u_{\ell,i}\circ\eta=u_{\ell',i}\circ\eta'.$$
		Then by the definition of $u_{\ell,i}$ and $u_{\ell{'},i}$, we have $$\textnormal{Id}_{Y}\circ \eta= \ell'_i\circ h\circ \rho \circ \eta'.$$
		Note that by the definition of $\sigma$, $\rho=\sigma\circ f_{\ell'}$. So $$\eta=\ell'_{i}\circ h\circ \sigma \circ f_{\ell'}\circ \eta'.$$ Thus by the definition of $\delta$, $$\eta=\ell'_{i}\circ h\circ \sigma\circ \delta.$$
		We can construct an analogous argument by replacing $i$ with $j$. So by symmetry, $$\eta'=\ell_{j}\circ h\circ \sigma\circ \delta.$$
		Now since $\ell'_i$ and $\ell_j$ are homomorphisms, $$R^\mathsf{Y}(\eta)\Leftrightarrow R^{\mathsf{X}}(h\circ \sigma\circ \delta)\Leftrightarrow R^\mathsf{Y}(\eta').$$
		So since $R^\mathcal{Y}(\eta')$ holds by assumption, $R^\mathsf{Y}(\eta)$ holds. Thus $(\mathsf{Z},f_{\ell},f_{\bar{e}})_{\ell,\bar{e}\in \textnormal{Ob}(J)}$ is a cocone of the line diagram in $Bl_{i_0}$.

		Next we show that $(\mathcal{Z},f_{\ell},f_{\bar{e}})_{\ell,\bar{e}\in \textnormal{Ob}(J)}$ is the colimit over the line diagram in $Bl_{i_0}$. If $
		(\mathcal{W},g_{\ell},g_{\bar{e}})_{\ell,\bar{e}\in \textnormal{Ob}(J)}$ where $\mathcal{W}=(\mathsf{W},\tau)$ is a cocone over the line diagram $G$ in $Bl_{i_0}$, by the definition of $F^{\mathsf{Z}}$ and $\rho$ there is a unique morphism of cocones $f$ from $(Z,f_{\ell},f_{\bar{e}})_{\ell, \bar{e}\in\text{Ob}(J)}$ to $
		(\mathcal{W},g_{\ell},g_{\bar{e}})_{\ell,\bar{e}\in \textnormal{Ob}(J)}$ that preserves function symbols and so that $\rho=\tau\circ f$. Thus all that we need to show is that $f$ preserves relation symbols. Fix a relation symbol $R$ of arity $r$ in $\mathcal{L}$. Then by the definition of $R^{\mathsf{Z}}$ for any $\delta\in\textnormal{Hom}_{\mathcal{C}}(r,Z)$, $R^{\mathsf{Z}}(\delta)$ holds if and only if there is a line $\ell$ and $\eta\in \textnormal{Hom}_{\mathcal{C}}(r,Y)$ so that $f_{\ell}\circ\eta=\delta$ and $R^{\mathsf{Y}}(\eta)$. Then since $g_{\ell}$ is a homomorphism $R^{\mathsf{Y}}(\eta)$ if and only if $R^{\mathsf{W}}(g_{\ell}\circ \delta)$. Thus $R^{\mathsf{Z}}(\delta)$ if and only if $R^{\mathsf{W}}(f\circ f_{\ell}\circ \eta)=R^{\mathsf{W}}(f\circ \delta)$. Therefore $f\in\text{Hom}_{Bl_{i_0}}(\mathcal{W},\mathcal{Z})$, so  $(\mathcal{Z},f_{\ell},f_{\bar{e}})_{\ell,\bar{e}\in \textnormal{Ob}(J)}$ is the colimit over the line diagram in $Bl_{i_0}$.

	\end{proof}
	\section{Application to Ramsey theory}
	In this section, we apply Theorem 1 to obtain results in Ramsey theory. First we define the Ramsey property for categories. Then we give a general Transfer Lemma that uses cocones to transfer Ramsey properties between categories. Next we state the Partite Lemma and show how the Partite Lemma follows directly from Theorem 1 and the Transfer Lemma. We then give a categorical version of the Partite Construction. We apply our Partite Construction to prove the results in \cite{6} and \cite{5} in a unified manner.
	\subsection {Ramsey property}
	In this section we will give our notation for the standard ideas found in a categorical approach to Ramsey theory.
	\begin{definition}
		For all $r>0$, an $r$-\textbf{coloring} of a set $S$ is a function $\chi\colon S\to r$ where $r=\{0,\dots, r-1\}$. Any $R\subseteq S$ is $\chi$-\textbf{monochromatic} if $\chi(R)=\{i\}$ for some $i\in r$.
	\end{definition}
	Using the above notation we define the main property we consider.
	\begin{definition}
		Fix a category $\mathcal{C}$, if $A,B,C\in \textnormal{Ob}(\mathcal{C})$ and $r>0$, then we say $C$ is a Ramsey witness for $A$ and $B$ (denoted $C\to (B)^{A}_{r}$) if for any $r$-coloring $\chi$ of $\textnormal{Hom}_{\mathcal{C}}(A,C)$ there is  $f\in \textnormal{Hom}_{\mathcal{C}}(B,C)$ so that $f\circ \textnormal{Hom}_{\mathcal{C}}(A,B)$ is $\chi$-monochromatic. 
		
		A category $\mathcal{C}$ has the Ramsey property if for all $A,B\in \textnormal{Ob}(\mathcal{C})$ and for all $r>0$, there is $C\in \textnormal{Ob}(\mathcal{C})$ so that $C\to (B)^{A}_{r}$.
	\end{definition}

The standard example of a category with the Ramsey property is the category $(\textbf{Fin},\leq)$ whose objects are finite linear orders and where morphisms are increasing injections. The fact that $(\textbf{Fin},\leq)$ has the Ramsey property is equivalent to Ramsey's Theorem.
	\subsection{Transferring Ramsey Property Over Cocones}
	
Given a category $\mathcal{C}$ with the Ramsey property, a category $\mathcal{D}$, and a map $F\colon \mathcal{C}\to \mathcal{D}$ it is natural to consider when $\mathcal{D}$ to has the Ramsey property. This situation has already been examined in \cite[Proposition 6.4]{12} and in \cite[Lemma 3.1]{7}. In fact our Transfer Lemma is equivalent to ideas found in \cite[Proposition 6.4]{12}. The difference between these theorems and our Transfer Lemma is that we use the established idea of a cocone to transfer a Ramsey statement.

 In \cite[Lemma 3.1]{7} the author shows that if there is a certain map $\mathcal{C}\to \mathcal{D}$ and $\mathcal{C}$ has the Ramsey property then $\mathcal{D}$ has the Ramsey property, we will not be taking such a global approach. More precisely, if $A,B\in \text{Ob}(\mathcal{C})$, $r>0$, there is a $C\in \text{Ob}(\mathcal{C})$, $D,E\in \text{Ob}(\mathcal{D})$, and a certain $F\colon \text{Hom}_{\mathcal{C}}(A,B)\to \text{Hom}_{\mathcal{D}}(D,E)$ then there is a Ramsey witness for $D$ and $E$ in $\mathcal{D}$. This local approach allows us to use the Ramsey property of many different categories $\mathcal{C}$ to prove that our target category $\mathcal{D}$ has the Ramsey property. Our Transfer Lemma will show that there is a Ramsey witness for $D$ and $E$ if $F$ is surjective and there is a cocone in $\mathcal{D}$ over a certain diagram. First we define this diagram and then we prove the Transfer Lemma.
	\begin{definition}Let $\mathcal{C},\mathcal{D}$ be categories. Suppose there are $A,B,C\in \textnormal{Ob}(\mathcal{C})$, $D,E\in \textnormal{Ob}(\mathcal{D})$ and $F\colon \textnormal{Hom}_{\mathcal{C}}(A,B)\to \textnormal{Hom}_\mathcal{D}(D,E)$.
		
		Let $J$ be the category with $$\text{Ob}(J)=\textnormal{Hom}_\mathcal{C}(A,C)\cup\textnormal{Hom}_\mathcal{C}(B,C)$$ and the only non-identity morphisms in $J$ are of the form \\$\text{Hom}_{J}(h,g)=\{(h,g,f)\colon f\in \text{Hom}_{\mathcal{C}}(A,B)\text{ and } g\circ f=h\}$ where $h\in \text{Hom}_{\mathcal{C}}(A,C)$ and $g\in \text{Hom}_{\mathcal{C}}(B,C)$.
		
		Then the \textbf{transfer diagram} $H\colon J\to \mathcal{D}$ is defined on objects by $$H(f)=D \text{ for all }f\in \textnormal{Hom}_\mathcal{C}(A,C), \ H(g)=E\text{ for all }g\in \textnormal{Hom}_\mathcal{C}(B,C)$$ 
		and on non-identity morphisms by 
	 $H(h,g,f)=F(f).$
	\end{definition}
	
	Now that we have defined the transfer diagram we can state the Transfer Lemma.
	\begin{lemma}[Transfer Lemma]
		Fix $\mathcal{C},\mathcal{D}$ be categories and $r>0$. Suppose there are $A,B,C\in \textnormal{Ob}(\mathcal{C})$ so that $C\to (B)^{A}_r$, $D,E\in \textnormal{Ob}(\mathcal{D})$ and a surjection $F\colon \textnormal{Hom}_{\mathcal{C}}(A,B)\to \textnormal{Hom}_\mathcal{D}(D,E)$. If $\mathcal{D}$ has a cocone $W\in\textnormal{Ob}(\mathcal{D}$) over the transfer diagram, then  $W\to (E)^D_r$ in the category $\mathcal{D}$.
	\end{lemma}
	\begin{proof}
		Let $A,B\in \mathcal{C}$, $D,E\in \mathcal{D}$, and $F\colon \textnormal{Hom}_{\mathcal{C}}(A,B)\to \text{Hom}_{\mathcal{D}}(D,E)$. Fix a cocone $(W,\phi_f)_{f\in \text{Ob}(J)}$ over the transfer diagram in $\mathcal{D}$. Thus by the definition of cocone we have the following commutative diagram,
		\begin{center}\begin{tikzcd}[column sep=small] &W \\ E_{g} \arrow[ur, swap, "\phi_g"] &\cdots& E_{g'} \arrow[ul, "\phi_{g'}"] \\ D_h \arrow[uur, bend left=60, "\phi_{f}", near start]\arrow[u, swap, "F(f)"]  &\cdots & D_{h'} \arrow[u, "F(f')"]\arrow[uul, bend right=60, swap, "\phi_{f'}", near start] \end{tikzcd}\end{center}
		
		Where we denote $H(g)$ by $E_g$ and $H(h)$ by $D_h$. We will use the above diagram commuting to show that  $W\to (E)^D_r$. \\
		
		Let $\chi\colon \text{Hom}_{\mathcal{D}}(D,W)\to r$ be a coloring. We define a coloring $\chi'\colon \text{Hom}_{\mathcal{C}}(A,C)\to r$ by  $\chi'(f)=\chi(\phi_f)$.

		Since $C\to (B)^{A}_r$, there is $g\in \text{Hom}_{\mathcal{C}}(B,C)$ so that $g\circ \text{Hom}_\mathcal{C}(A,B) \text{ is }\chi'  \text{-monochromatic}.$
		It remains to show that $\phi_g \circ \text{Hom}_\mathcal{D}(D,E) \text{ is }\chi \text{-monochromatic}$.
		
		Let $j\in \text{Hom}_\mathcal{D}(D,E)$, since $F$ is a surjection there is $h\in \text{Hom}_\mathcal{C}(A,B)$ so that $F(h)=j$. So by the definition of cocone, $$\phi_{g\circ h}=\phi_g\circ F(h)=\phi_g\circ j.$$
		Then by the definition of $\chi'$,
		$$\chi(\phi_g\circ j)=\chi(\phi_{g\circ h})=\chi'(g\circ h).$$
		Because $g\in \text{Hom}_{\mathcal{C}}(B,C)$ is $\chi'$-monochromatic, $\phi_g \circ \text{Hom}_\mathcal{D}(D,E) \text{ is }\chi \text{-monochromatic}.$
	\end{proof}
	
	\subsection{The Partite Lemma}
	In this section, we will state and prove the Partite Lemma. We show that Theorem 1 gives us precisely what is necessary to apply the Transfer lemma to the Hales--Jewett Theorem which will prove the Partite Lemma. In order to use the Transfer Lemma we define a category which we call the Hales--Jewett category and give a reformulation of the Hales--Jewett Theorem using the Hales--Jewett category. Then we prove the Partite Lemma by showing that a line diagram is a transfer diagram for the Hales--Jewett category.

	Fix a finite set $P$. We let $HJ(P)$ be the category with $\textnormal{Ob(HJ(P))}=\mathbb{N}$ and\\ $\text{Hom}_{\textnormal{HJ(P)}}(0,N)=P^N$, in particular $\text{Hom}_{\textnormal{HJ(P)}}(0,1)=P$. Define $\textnormal{Hom}_{\textnormal{HJ(P)}}(1,N)$ as the set of lines in $P^N$ and let all other morphisms be identities.   For all $p\in P$, $\ell$ a line in $P^N$ we define $\ell\circ p\in P^N$ by,
	$$(\ell\circ p)_i=\left\{\begin{array}{ll}
		p & \text{ if } i\in d(\ell)\\
		\ell_i & \text{otherwise}\\
	\end{array} \right.$$
	Note that for any line $\ell$ in $P^N$ and any $\bar{e}\in P^N$, $\bar{e}\in \ell$ if and only if $\ell\circ \ell(\bar{e})=\bar{e}$ in $HJ(P)$.
	
	In \cite{7} the author defines the Graham-Rothschild category, and the category HJ(P) fits nicely as a subcategory of the Graham-Rothschild category with some modification.
	With our terminology we can now give a reformulation of the Hales--Jewett Theorem.
	\begin{theorem}[Hales--Jewett] For all $r>0$ and for each finite set $P$, there is $N$ so that $N\to (1)^0_r$ in \textnormal{HJ(P)}. 
	\end{theorem}
	
 The version of the Partite Lemma below is a unification of the main lemmas in \cite{6} and \cite {5}. 
	
	\begin{corollary}[Partite Lemma]
		Let $\mathcal{C}$ be a category so that for all $X,Y\in \textnormal{Ob}(\mathcal{C})$,\\ $\textnormal{Hom}_{\mathcal{C}}(X,Y)$ is finite and suppose that $\mathcal{C}$ has colimits over all diagrams $G\colon J\to \mathcal{C}$ where the index category $J$ is finite. Fix a $\mathcal{C}$-language $\mathcal{L}$ and let
		$i_0\in \text{Hom}(\mathcal{C})$ have a left inverse. Then for any $r>0$, any monic domain object $\mathcal{X}\in \text{Ob}(Bl^{m}_{i_0})$, and any codomain object $\mathcal{Y}\in \text{Ob}(Bl^{m}_{i_0})$ there is a $\mathcal{Z}\in \text{Ob}(Bl^{m}_{i_0})$ so that $\mathcal{Z}\to (\mathcal{Y})^{\mathcal{X}}_r$.
	\end{corollary}
	\begin{proof}
		Fix $r>0$, $\mathcal{X}$ a monic domain object in $Bl^{m}_{i_0}$, and $\mathcal{Y}$ a codomain object in $Bl^{m}_{i_0}$. First we show that we can still apply Theorem 1 even though $\mathcal{C}$ no longer has all colimits. Note that the proof of Theorem 1 only used the fact that $\mathcal{C}$ had colimits over diagrams with index category $J$ where the $\text{Ob}(J)=\text{Hom}_{Bl^{m}_{i_0}}(\mathcal{X},\mathcal{Y})^N\cup \{\ell\colon \ell \text{ is a line in }\text{Hom}_{Bl^{m}_{i_0}}(\mathcal{X},\mathcal{Y})^N\}$ for some $N$. Since by assumption on $\mathcal{C}$ the set $\textnormal{Hom}_{\mathcal{C}}(X,Y)$ is finite, $\textnormal{Hom}_{Bl^{m}_{i_0}}(\mathcal{X},\mathcal{Y})$ is finite. Thus the set of objects in $J$ is finite, so we can apply Theorem 1. Thus all line diagrams have cocones in $Bl_{i_0}$.\\
	 If $P=\textnormal{Hom}_{Bl^{m}_{i_0}}(\mathcal{X},\mathcal{Y},\rho)$, then $P$ is finite so by the Hales--Jewett Theorem there is a $N\to (1)^{0}_{r}$ in $HJ(P)$. Our goal is to apply the Transfer Lemma to $\text{Id}\colon \text{Hom}_{HJ(P)}(0,1)\to \text{Hom}_{Bl^{m}_{i_0}}(\mathcal{X},\mathcal{Y})$. To do so we show the transfer diagram  is the line diagram. Note that the transfer diagram has index $J$ where $$\text{Ob}(J)=\text{Hom}_{HJ(P)}(0,N)\cup \text{Hom}_{HJ(P)}(1,N)=P^N\cup \{\ell\colon \ell\text{ is a line in }P^N\}$$ and the non-identity morphisms are of the form \\$\text{Hom}_{J}(\bar{e},\ell)=\{(\bar{e},\ell,p)\colon p\in P\text{ so that } \ell\circ p=\bar{e}\}$ where $\bar{e}\in P^N$ and $\ell$ is a line in $P^N$. Then for all lines $\ell$ and $\bar{e}\in P^N$, if there is a $p\in P$ such that $p\circ \ell=\bar{e}$, then $\bar{e}\in \ell$ and $\ell(\bar{e})=p$. Thus the only non-identity morphisms in $J$ are $\text{Hom}_{J}(\bar{e},\ell)=\{(\bar{e},\ell,\ell(\bar{e}))\}$ if $\bar{e}\in \ell$. Then the transfer diagram is $G\colon J\to Bl^{m}_{i_0}$ defined by  $$G(\bar{e})=
		\mathcal{X} \textnormal{ if }\bar{e}\in  P^N$$
		$$G(\ell)=\mathcal{Y} \text{ if }\ell \text{ is a line in }P^N $$
		and on non-identity morphisms  $G(\bar{e},\ell,\ell(\bar{e}))=\ell(\bar{e})$. Thus the transfer diagram is the line diagram.
		
		. \end{proof}
	\subsection{The Partite Construction}
	In this section, we expand the Partite Lemma as stated in Section 5.3 from a result in $Bl^{m}_{i_0}$ to a larger category that we call $Bl_{\mathcal{D}}$. The category $Bl_{\mathcal{D}}$, that we define precisely below, is a subcategory of $Bl$. We need to cut down from $Bl$ to $Bl_{\mathcal{D}}$ since $Bl$ does not have an analog of the Partite Lemma while we prove a version of the Partite Lemma for $Bl_{\mathcal{D}}$ below. To define $Bl_{\mathcal{D}}$ we consider a new category $\mathcal{D}$ that will have the Ramsey property and a functor $G\colon \mathcal{D}\to \mathcal{C}$.  This category $\mathcal{D}$ is analogous to the category of finite linear orders in the Ne\v{s}et\v{r}il --R\"{o}dl Theorem.  
	\begin{definition}
		Let $\mathcal{C},\mathcal{D}$ be categories, $G\colon \mathcal{D}\to \mathcal{C}$ be a functor, and $\mathcal{L}$ be a $\mathcal{C}$-language.  A $\mathcal{D}$-\textbf{block} is a triple $\mathtt{X}=(\mathsf{X},\pi,K)$ where $\mathsf{X}$ is an $\mathcal{L}$-structure, $K\in\textnormal{Ob}(\mathcal{D})$, and $\pi\in \textnormal{Hom}_{\mathcal{C}}(X,G(K))$.\end{definition} 
	For ease of notation if $\mathcal{X}=(\mathsf{X},\pi)$ we denote the $\mathcal{D}$-block $(\mathsf{X},\pi,K)$ by $(\mathcal{X},K)$. We will also use the letters $\mathtt{X},\mathtt{Y},\mathtt{Z}$ to denote $\mathcal{D}$-blocks with first coordinate $\mathcal{X},\mathcal{Y},\mathcal{Z}$.
	
	We define a category $Bl^{m}_{\mathcal{D}}$ of $\mathcal{D}$-blocks. Let $\text{Ob}(Bl^{m}_{\mathcal{D}})$ be $\mathcal{D}$-blocks and if\\ $\mathtt{X}=(\mathcal{X},K),\mathtt{Y}=(\mathcal{Y},L)\in \text{Ob}(Bl_{\mathcal{D}})$, 
	then $$\textnormal{Hom}_{Bl_{\mathcal{D}}}(\mathtt{X},\mathtt{Y})=\{f\in \textnormal{Hom}_{\mathcal{C}}(X,Y)\colon \text{there is a }i\in \text{Hom}_{\mathcal{D}}(K,L) \text{ so that }f\in \text{Hom}(Bl^{m}_{G(i)}) \}$$

	We now  expand the Partite Lemma to the category $Bl^{m}_{\mathcal{D}}$.
	\begin{corollary}[Partite Construction]
		Let $\mathcal{C}$ and $\mathcal{D}$ be categories and let $G\colon \mathcal{D}\to \mathcal{C}$ be a functor so that the following hold: 
		\begin{enumerate}
			\item For all $X,Y\in \textnormal{Ob}(\mathcal{C})$, $\textnormal{Hom}_{\mathcal{C}}(X,Y)$ is finite. Similarly, for all $K,L\in \textnormal{Ob}(\mathcal{D})$, $\textnormal{Hom}_{\mathcal{D}}(K,L)$ is finite. 	
			
			\item For all $f\in \textnormal{Hom}(\mathcal{D})$, $G(f)$ has a left inverse in $\mathcal{C}$.
			\item If $F\colon J\to \mathcal{C}$ is a functor where $\textnormal{Ob} (J)$ is finite, then $\mathcal{C}$ has a colimit over $F$.
			\item $\mathcal{D}$ has the Ramsey property.
		\end{enumerate}	 
	 Fix a $\mathcal{C}$-language $\mathcal{L}$. Then for any $r>0$, any monic  $\mathtt{X}\in \text{Ob}(Bl^{m}_{\mathcal{D}})$, and any $\mathtt{Y}\in \text{Ob}(Bl^{m}_{\mathcal{D}})$ there is a $\mathtt{Z}\in \text{Ob}(Bl^{m}_{\mathcal{D}})$ so that $\mathtt{Z}\to (\mathtt{Y})^{\mathtt{X}}_r$.	
	\end{corollary}
	The proof of Corollary 5 follows ideas standard in the field of structural Ramsey theory, for example see \cite{10}, though we specifically follow the formulations in \cite{6} and \cite{5}.
	\begin{proof}
		Let $\mathtt{X}=(\mathsf{X},\pi,K)\in \textnormal{Ob}(Bl_{\mathcal{D}})$ be monic,  
		$\mathtt{Y}=(\mathsf{Y},\rho,L)\in\textnormal{Ob}(Bl_{\mathcal{D}})$, and let $r>0$. By the Ramsey property of $\mathcal{D}$ there is $M\in \mathcal{D}$ so that $M\to (L)^{K}_r$ in $\mathcal{D}$. We now define a block $(\mathcal{Y}_0,\sigma_0,M)$ where $Y_0$ is defined as the categorical disjoint sum of copies of $Y$.  In particular let $J$ be the category with $\text{Ob}(J)=\text{Hom}_{\mathcal{D}}(L,M)$ and the only  morphisms of $J$ are identities. Let $H\colon J\to \mathcal{C}$ be defined by $H(i)=Y$ for all $i\in \text{Ob}_J$. Let $(Y_0,h_i)_{i\in \text{Hom}_{\mathcal{D}}(L,M)}$ be the colimit over $H$ in $\mathcal{C}$. 
		
		Note that since the only morphisms in $J$ are identities, for any $W\in \text{Ob}(\mathcal{C})$ and any collection $(f_i)_{i\in\text{Hom}_{\mathcal{D}(L,M)}}$ where $f_i\in \text{Hom}_{\mathcal{C}}(Y,W)$, $(W,f_i)_{i\in \text{Hom}_{\mathcal{D}}(L,M)}$ is a cocone over $H$. In particular $(G(M),G(i)\circ \rho)_{i\in \text{Hom}_{\mathcal{D}}(L,M)}$ is a cocone over $H$. Thus by the definition of colimit there is $\rho_0\in \text{Hom}_{\mathcal{C}}(Y_0,G(M))$ so that $\rho_0\circ h_i=G(i)$ for all $i\in \text{Hom}_{\mathcal{D}}(L,M)$.
		
		Let $\gamma\in\text{Hom}_{\mathcal{C}}(Y_0,r)$ and $F$ be a function symbol in $\mathcal{L}$ of arity $(q,s)$, then consider the cocone $(s,F^{\mathcal{Y}}(\gamma\circ h_i))_{i\in \text{Hom}_{\mathcal{D}}(L,M)}$. By the definition of colimit there is a $F^{\mathsf{Y}_0}(\gamma)$ so that $F^{\mathsf{Y}_0}(\gamma)\circ h_i=F^{\mathsf{Y}}(\gamma\circ h_i)$ for all $i\in \text{Hom}_{\mathcal{D}}(L,M)$. Now given a relation symbol $R$ in $\mathcal{L}$ of arity $q$ and $\delta\in \text{Hom}_{\mathcal{C}}(q,Y_0)$ define $R^{\mathsf{Y}_0}$ so that, $$R^{\mathsf{Y}_0}(\delta)\Leftrightarrow \text{ there are }\eta\in \text{Hom}_{\mathcal{C}}(q,Y),i\in \text{Hom}_{\mathcal{D}}(L,M) \text{ so that }h\delta=h_i\circ \eta \text{ and }R^{\mathcal{Y}}(\eta).$$ 
		It is easy to check that each $h_i$ is a $G(i)$-monomorphism.\\ 
		
		We enumerate $\text{Hom}_{\mathcal{D}}(K,L)$ by letting $\text{Hom}_{\mathcal{D}}(K,M)=\{j_k\colon k<n\}$. Then recursively define $(\mathcal{Y}_k,M)$ by the Partite Lemma so that $\mathcal{Y}_{k+1}\to (\mathcal{Y}_k)^{\mathcal{X}}_r$ in $Bl_{j_k}$. We show that $\mathsf{Z}=(\mathcal{Y}_{n},M)$ satisfies the Ramsey property.\\
		
		Fix a coloring $\chi\colon \text{Hom}_{Bl_{\mathcal{D}}}(\mathsf{X},\mathsf{Z})\to r$. We define morphisms $g_k$ recursively by the Partite Lemma so that $g_{n}\circ \cdots\circ g_{k}\circ \text{Hom}_{Bl_{j_k}}(\mathcal{X},\mathcal{Y}_k)$ is $\chi$-monochromatic.  We claim that if $g=g_n\circ \cdots \circ g_1$ then for all $f\in \text{Hom}_{Bl_{\mathcal{D}}}(G(K),Y_0)$, $\chi(g\circ f)$ depends only on the $j_k\colon K\to M$ so that $\rho_0\circ f=i_k\circ \pi$. To prove the claim suppose that $f$ is a $j_k$-monomorphism. 
		So since $g_l$ is a $\text{Id}_M$-monomorphism for all $l<k$, $g_{k-1}\cdots g_1(f)$ is a $j_k$-monomorphism. Then by the definition of $g_k$, $g_n\cdots g_k \circ g_{k-1}\cdots g_1\circ f$ is a fixed color which proves the claim.  Then define a coloring $\chi'$ so that $$\chi'\colon \text{Hom}_{\mathcal{D}}(K,M)\to r, \ \chi'(j_k)=\chi(g\circ f)\text { if }\rho_0\circ f=i_k\circ \pi$$
		
		Then by the Ramsey property of $\mathcal{D}$ there is $i_0\in \text{Hom}_{\mathcal{D}}(L,M)$ so that $i_0\circ \text{Hom}_{\mathcal{D}}(K,L)$ is $\chi'$-monochromatic. Then since $h_{i_0}$ is a $G(i_0)$-monomorphism, $g_n\circ \cdots\circ g_0\circ h_{i_0}$ witnesses that $\mathtt{Z}\to (\mathtt{Y})^{\mathtt{X}}_r$. \end{proof}
	\subsection{The Theorems of Solecki}
	We will give the results of \cite{6} and \cite{5} as corollaries of the Partite Construction. Thus we will show that the results in \cite{6} and \cite{5} can be proven in a unified manner.
	
	First we give some notation so that we can reformulate the results in \cite{6} and \cite{5} into our context. Let $\mathcal{C},\mathcal{D}$ be categories, $G\colon \mathcal{D}\to \mathcal{C}$ be a functor, and $\mathcal{L}$ be a $\mathcal{C}$-language. Then $(\mathcal{L},\mathcal{D})$ is the category where objects are $\mathcal{L}$-structures of the form $G(K)$ where $K\in \mathcal{D}$ which we denote by $\mathsf{K}$ and if $\mathsf{K},\mathsf{M}\in \textnormal{Ob}(\mathcal{L},\mathcal{D})$ then $$\text{Hom}_{(\mathcal{L},\mathcal{D})}(\mathsf{K},\mathsf{M})=\{G(i)\colon i\in \text{Hom}_{\mathcal{D}}(K,M)\text{ and }G(i)\text{ is a homomorphism}\}$$
	With this notation we state and prove the main result in \cite{5}.
	\begin{corollary}[Solecki,\cite{5}]
		Let $G\colon (\textnormal{\textbf{Fin}},\leq)\to \textnormal{\textbf{Fin}}$ be defined by $G(K,\leq)=K$ and $G(f)=f$. For any $\textnormal{\textbf{Fin}}$-language $\mathcal{L}$ the category $(\mathcal{L},(\textnormal{\textbf{Fin}},\leq))$ has the Ramsey property.
	\end{corollary}
	Note that the category $(\mathcal{L},(\textnormal{\textbf{Fin}},\leq))$ is the category whose objects are linearly ordered structures and morphisms are increasing injective homomorphisms. Thus Corollary 6 is an expansion of the Ne\v{s}et\v{r}il--R\"{o}dl Theorem.
	\begin{proof}
		Let $\mathsf{K},\mathsf{M}\in \textnormal{Ob}(\mathcal{L},(\textbf{Fin},\leq))$ and $r>0$.  We view $\mathsf{K}$ and $\mathsf{M}$ as the $(\textbf{Fin},\leq)$-blocks $\mathsf{K}=(\mathsf{K},\text{Id}_{G(K)},K)$ and $\mathsf{M}=(\mathsf{M},\text{Id}_M,M)$ respectively.  Next we show that $\textbf{Fin}$, $(\textbf{Fin},\leq)$, and $G$ satisfy conditions (1)-(4) for the Partite Construction. The only property that is not clear is that the category \textbf{Fin} has colimits over diagrams where the index category $J$ has a finite set of objects. This is a standard result in category theory. The colimit over diagram $F$, is $$Z=\bigsqcup_{S\in \text{Ob}(J)}F(S)/_{\sim'}$$ where $\sim'$ is the transitive closure of the relation given by $$(s,F(S))\sim(t, F(T))\Leftrightarrow \text{ there is an } f\in J,f\colon S\to T, F(f)(s)=t,$$
		and $\phi_S\colon F(S)\to Z$ is given by $\phi_S(s)=[s]$.\\
		
		By the Partite Construction there is a $\mathtt{Z}=(\mathsf{Z},\pi,L)$ so that $\mathtt{Z}\to (\mathsf{M})^{\mathsf{K}}_r$ in $Bl_{(\textbf{Fin},\leq)}$. We order $Z$ so that $\pi$ becomes a weakly increasing map, then 
		 $\mathsf{Z}\to (\mathsf{M})^{\mathsf{K}}_r$ in $(\mathcal{L},(\textbf{Fin},\leq))$.
	\end{proof}
	For the main result in \cite{6} we need to define another category. \textnormal{(\textbf{Fin}},$\leq^{*}$) is the category where $\textnormal{Ob}((\textnormal{\textbf{Fin}},\leq^{*}))$ are finite linear orders and for all $L,K\in \textnormal{Ob}((\textnormal{\textbf{Fin}},\leq^{*}))$, $$\textnormal{Hom}_{(\textnormal{\textbf{Fin}},\leq^{*})}(L,K)=\{f\in K^L\colon \ f\text{ is a rigid surjection}\}.$$	Where
	a \textbf{rigid surjection} is a map $s\colon L\to K$ between linear orders that is a surjection and images of initial segments of $L$ are initial segments of $K$. This definition allows us to state the main result in \cite{6}.
	\begin{corollary}[Solecki,\cite{6}]
		Let $G\colon (\textnormal{\textbf{Fin}},\leq^*)^{\textnormal{op}}\to \textnormal{\textbf{Fin}}^{\textnormal{op}}$ be defined by $G(K,\leq)=K$ and $G(f)=f$. For any $\textnormal{\textbf{Fin}}^{\textnormal{op}}$-language $\mathcal{L}$ the category $(\mathcal{L},(\textnormal{\textbf{Fin}},\leq^*)^{\textnormal{op}})$ has the Ramsey property.
	\end{corollary}
	\begin{proof}
		Let $\mathsf{K},\mathsf{M}\in \textnormal{Ob}(\mathcal{L},(\textbf{Fin},\leq^*)^{\text{op}})$ and $r>0$.  We view $\mathsf{K}$ and $\mathsf{M}$ as the $(\textbf{Fin},\leq^*)^{\text{op}}$-blocks $\mathsf{K}=(\mathsf{K},\text{Id}_{G(K)},K)$ and $\mathsf{M}=(\mathsf{M},\text{Id}_M,M)$ respectively.  We claim that $\textbf{Fin}^{\text{op}}$, $(\textbf{Fin},\leq^*)^{\text{op}}$, and $G$ satisfy conditions (1)-(4) for the Partite Construction.   First note that $\textbf{Fin}^{\text{op}}$ has colimts when the category $J$ has a finite set of objects. The limit over diagram $F$ is $$Z=\{(s_S)_{s\in J}\in \prod_{S\in \text{Ob}(J)}F(S)\colon \text{ for all } S,T\in \text{Ob}(J), f\in \text{Hom}_J(X,Y), F(f)(s_S)=(s_T)\}$$
		and $\phi_S\colon Z\to F(S)$ is given by the projection maps $\pi_S$.  Also the category $(\textbf{Fin},\leq^*)^\text{op}$ has the Ramsey by the Graham--Rothschild Theorem in \cite{8}. The remaining conditions are trivial to prove. By the Partite Construction there is a $\mathtt{Z}=(\mathsf{Z},\rho,L)$ so that $\mathtt{Z}\to (\mathsf{M})^{\mathsf{K}}_r$ in
		$Bl_{(\textbf{Fin},\leq^*)^{\text{op}}}$. We linearly order $\rho(L)$ by $a\leq b$ if and only if $\text{min}(\rho^{-1}(a))\leq \text{min}(\rho^{-1}(b))$ and order the rest of $Z$ so that $\rho(L)$ is an initial segment. Then $\mathsf{Z}\to (\mathsf{M})^{\mathsf{K}}_r$ in \\
		$(\mathcal{L},(\textbf{Fin},\leq^*)^{\text{op}})$.
	\end{proof}
\section{Declarations}
\subsection{Conflicts of Interest}
Not applicable.
\subsection{Availability of Data and Materials}
Not applicable.
\subsection{Funding}
The author was partially supported by NSF grant DMS-1954069.
	\bibliographystyle{plain}
	\bibliography{reference}
\end{document}